\documentclass[reqno]{amsart}
\usepackage{amsmath,amssymb}
\usepackage{color}

\newcommand\fuc{Fu\v c\'\i k\ }

\newtheorem{thm}{Theorem}[section]
\newtheorem{prop}[thm]{Proposition}
\newtheorem{lemma}[thm]{Lemma}
\newtheorem{cor}[thm]{Corollary}

\theoremstyle{definition}

\theoremstyle{remark}
\newtheorem{remark}[thm]{Remark}

\newcounter{myacount}
\newenvironment{alist}
{  \begin{list}
   {$(\alph{myacount})$}
   {\usecounter{myacount}\setlength\labelwidth{10 cm}}  }
{\end{list}}
\def\bal{\begin{alist}} \def\eal{\end{alist}}
\newcounter{myicount}
\newenvironment{ilist}
{  \begin{list}
   {$(\roman{myicount})$}
   {\usecounter{myicount}\setlength\labelwidth{10 cm}}  }
{\end{list}}
\def\bil{\begin{ilist}} \def\eil{\end{ilist}}

\newenvironment{tlist}
{  \begin{list}
   {\large$\blacktriangleright$}
   {\usecounter{mytbcount}\setlength\labelwidth{10 cm}}  }
{\end{list}}
\def\btl{\begin{tlist}} \def\etl{\end{tlist}}

\newenvironment{blist}
{  \begin{list}
   {\large$\bullet$}
   {\usecounter{mytbcount}\setlength\labelwidth{10 cm}}
   }
{\end{list}}
\def\bbl{\begin{blist}} \def\ebl{\end{blist}}

\numberwithin{equation}{section}

\newcommand{\beq}{\begin{equation}}
\newcommand{\eeq}{\end{equation}}
\newcommand{\beqs}{\begin{equation*}}
\newcommand{\eeqs}{\end{equation*}}

\newcommand{\alns}[1]{\begin{align*} #1 \end{align*}}

\newcommand\D{\Delta}
\newcommand\al{\alpha}
\newcommand\be{\beta}

\newcommand\de{\delta}
\newcommand\De{\Delta}

\newcommand\la{\lambda}

\newcommand\Si{\Sigma}

\newcommand{\mat}[1]{\begin{pmatrix} #1 \end{pmatrix}}

\newcommand \bmat{\begin{pmatrix}}
\newcommand \emat{\end{pmatrix}}

\newcommand\tde{\widetilde \de}
\newcommand\tpsi{\widetilde \psi}
\newcommand\ttau{{\widetilde \tau}}

\newcommand\R{\mathbb{R}}

\let\le=\leqslant
\let\ge= \geqslant

\newcommand\X{\times}

\renewcommand\emptyset{\mbox{\Large \o}}
\newcommand{\diff}{\,\mathrm{d}}

\newcommand\sgn{{\rm sgn}\,}

\begin{document}

\title[Landesman-Lazer conditions]
{Landesman-Lazer conditions at half-eigenvalues
of the $p$-Laplacian}
\author{Fran\c cois Genoud, Bryan P. Rynne}
\address{Department of Mathematics and the Maxwell Institute for
Mathematical Sciences, Heriot-Watt University,
Edinburgh EH14 4AS, Scotland.}
\email{B.P.Rynne@hw.ac.uk}
\email{F.Genoud@hw.ac.uk}

\begin{abstract}
We study the existence of solutions of the Dirichlet
problem
\begin{gather}
-\phi_p(u')'
-a_+ \phi_p(u^+) + a_- \phi_p(u^-) -\la \phi_p(u)  = f(x,u) ,
\quad x \in (0,1), \label{pb.eq} \tag{1}\\
u(0)=u(1)=0,\label{pb_bc.eq} \tag{2}
\end{gather}
where  $p>1$, $\phi_p(s):=|s|^{p-1}\sgn s$ for  $s \in \R$,
the coefficients $a_\pm \in C^0[0,1]$, $\la \in \R$,
and $u^\pm := \max\{\pm u,0\}$.
We suppose that $f\in C^1([0,1]\X\R)$
and that there exists $f_\pm \in C^0[0,1]$ such that
$\lim_{\xi\to\pm\infty} f(x,\xi) = f_\pm(x)$, for all $x \in [0,1]$.
With these conditions the problem
\eqref{pb.eq}-\eqref{pb_bc.eq}
is said to have a `jumping
nonlinearity'.
We also suppose that the problem
\begin{gather}
-\phi_p(u')' =  a_+ \phi_p(u^+) - a_- \phi_p(u^-)  + \la \phi_p(u)
\quad\text{on} \ (0,1),  \tag{3}  \label{heval_pb.eq}
\end{gather}
together with \eqref{pb_bc.eq},
has a non-trivial solution $u$.
That is,  $\la$ is a `half-eigenvalue' of
\eqref{pb_bc.eq}-\eqref{heval_pb.eq},
and the problem
\eqref{pb.eq}-\eqref{pb_bc.eq}
 is said to be `resonant'.
Combining a shooting method with so called `Landesman-Lazer' conditions,
we  show that the problem \eqref{pb.eq}-\eqref{pb_bc.eq} has a
solution.

Most previous existence results for jumping nonlinearity problems at
resonance have considered the case where the coefficients $a_\pm$ are
constants,
and the resonance has been at a point in the `\fuc spectrum'.
Even in this constant coefficient case our result extends previous
results.
In particular, previous variational approaches have required strong
conditions on the location of the resonant point,
whereas our result applies to any point in the \fuc spectrum.
\end{abstract}

\maketitle

\section{Introduction}  \label{intro.sec}

We consider the $p$-Laplacian Dirichlet problem
\begin{gather}
-\phi_p(u')'
-a_+ \phi_p(u^+) + a_- \phi_p(u^-) -\la \phi_p(u)  = f(x,u) ,
\quad x \in (0,1), \label{eq.eq}\\
u(0)=u(1)=0,\label{dbc.eq}
\end{gather}
where  $p>1$, $\phi_p(s):=|s|^{p-1}\sgn s$ for  $s \in \R$,
the coefficients $a_\pm \in C^0[0,1]$, $\la \in \R$,
and $u^\pm := \max\{\pm u,0\}$.
We assume that $f\in C^1([0,1]\X\R)$ satisfies the following hypotheses:
\medskip

\bbl
\item[$(f1)$]
there exist $f_\pm \in C^0[0,1]$  such that
\begin{equation}  \label{fpm.eq}
\lim_{\xi\to\pm\infty} f(x,\xi) = f_\pm(x) , \quad x\in[0,1] ,
\end{equation}
\item[$(f2)$]
there exist $K_1>0$ and $\rho \in [0,1)$, with $\rho>2-p$ if $p \le 2$,
such that
\begin{equation}  \label{fxi_bnd.eq}
|\xi^\rho f_\xi(x,\xi)| \le K_1 , \quad  x \in [0,1], \ \xi \in \R .
\end{equation}
\ebl
Of course, it follows immediately from $(f1)$ that there exists $K_0>0$
such that
\begin{equation}  \label{fint.eq}
|f(x,\xi)| \le K_0 , \quad  x \in [0,1], \ \xi \in \R .
\end{equation}

\noindent
These conditions imply that the asymptotic behaviour of
\eqref{eq.eq} as $u \to \pm \infty$ is determined by
the coefficients $a_\pm$ and $f_\pm$
(and the value of $\la$),
and that these asymptotic behaviours may be different.
Such problems are often termed {\em jumping}.

Let us now introduce some basic notations and definitions. The spaces
$C^i[0,1]$ will be endowed with their usual sup-type norms
$|\cdot|_i, \ i=0,1$.
We define $D_p$ to be the set of functions $u \in C^1[0,1]$ such that
$\phi_p(u')\in C^1[0,1]$, and
an operator $\D_p : D(\D_p) \to C^0[0,1]$ by
\begin{align*}
D(\De_p) &:= \{ u \in D_p : \text{$u$ satisfies \eqref{dbc.eq}}
\},
\\
\De_p(u) &:= \phi_p(u')' , \quad u \in D(\De_p).
\end{align*}
In addition, we denote by $u \to f(u)$ the Nemitskii mapping associated
with $f$,
that is, $f(u)$ is defined by $f(u)(x) := f(x,u(x)), \ x \in [0,1]$.
(We will use a similar notation for other Nemitskii mappings.)
Clearly, $u \to f(u)$ is a continuous mapping from $C^0[0,1]$ to
$C^0[0,1]$.
With this notation the problem \eqref{eq.eq}-\eqref{dbc.eq} can now be
rewritten as
\begin{equation}  \label{nlslvble.eq}
-\De_p(u) - a_+ \phi_p(u^+) + a_- \phi_p(u^-) - \la \phi_p(u) = f(u) ,
\quad  u \in D(\De_p) .
\end{equation}

Related to \eqref{nlslvble.eq} is the half-eigenvalue problem
\begin{equation}  \label{heval.eq}
-\D_p(u) -  a_+ \phi_p(u^+) + a_- \phi_p(u^-)  = \la \phi_p(u),
\quad  u \in D(\De_p) .
\end{equation}
The problem \eqref{heval.eq} is positively homogeneous, in the sense
that if $u$ is a solution then $tu$ is also a
solution for all $t \ge 0$.
Any $\la$ for which \eqref{heval.eq}
has a non-trivial solution $u$ is called a {\em half-eigenvalue},
and any corresponding non-trivial solution $u$ is a
{\em half-eigenfunction.}
We define a spectrum for \eqref{heval.eq} to be the set
$$
\Si_H = \Si_H(a_\pm) :=
\{ \la \in \R : \text{\eqref{heval.eq} has a non-trivial solution}\}.
$$

To describe the structure of the set $\Si_H$ we introduce some further
notation.
For each integer $k \ge 0$ and $\nu \in \{\pm\}$, let $S_{k,\nu}$
denote the set of functions $u \in D(\D_p)$ having only
simple zeros and exactly $k$ such zeros in $(0,1)$,
and such that $\nu u'(0) >0$.
Any non-trivial solution $u$ of \eqref{heval.eq} belongs to $S_{k,\nu}$
for some $k \ge 0$ and $\nu \in \{\pm\}$
(see Remark~\ref{wellposedness.rem} below).
The following description of the spectrum $\Si_H$ was given in
\cite{RYN2}.

\begin{thm} \label{heval.thm}
For each $k \ge 0$, \eqref{heval.eq} has unique solutions
$(\la,u)=(\la_{k,\pm},u_{k,\pm}) \in \R \X S_{k,\pm}$ with
$u_{k,\pm}'(0) = \pm 1$.
All the half-eigenfunctions corresponding to $\la_{k,\pm}$
are of the form $t u_{k,\pm}$, with $t>0$,
and the spectrum $\Si_H$ is given by
% $$
% \Si_H =  \{\la_{k,\pm} : k\ge0, \ \nu \in \{\pm\}\}.
% $$
$$
\Si_H =  \bigcup_{k \ge 0} \{\la_{k,\pm}\}.
$$
The half-eigenvalues are increasing, in the sense that
\begin{equation}  \label{bermonot.eq}
k' > k \ \Longrightarrow \ \la_{k',\nu'} > \la_{k,\nu},
\quad \text{for each $\nu',\,\nu \in \{\pm\}$},
\end{equation}
and $\lim_{k\to\infty} \la_{k,\nu}=\infty$.
Furthermore, $\pm u_{0,\pm}>0$ on $(0,1)$.
\end{thm}

It will also be convenient to introduce the following notation:
for each
$k \ge 0$, let
$$
\la_{k,\min} = \min \{ \la_{k,+}, \la_{k,-} \}\quad\text{and}\quad
\la_{k,\max} = \max \{ \la_{k,+}, \la_{k,-} \},
$$
and let $u_{k,\min}, \ u_{k,\max}$ denote the
corresponding half-eigenfunctions.

The discussion of solvability conditions for
\eqref{nlslvble.eq}
relies upon the position of $\la$ with respect to the spectrum $\Si_H$.
When $\la\not\in\Si_H$ the problem
\eqref{nlslvble.eq} is said to be {\em non-resonant},
and this case has been extensively studied,
see \cite{RYN2} and the references therein for more information.
Briefly, the main results are as follows:

\medskip
\noindent
$(a)$ \ if
$\la_{k,\max} < \la < \la_{k+1,\min}$, for some $k \ge 0$
(or $\la < \la_{0,\min}$),
then \eqref{nlslvble.eq} has a solution for any $f$
(see \cite[Theorem 4.1]{RYN2});

\medskip
\noindent
$(b)$ \ if
$\la_{k,\min} < \la < \la_{k,\max}$, for some $k \ge 0$,
then there exists $h \in C^0[0,1]$
(independent of $u$)
such that \eqref{nlslvble.eq}, with $f = h$, has no solution
(see \cite[Theorem~5.1]{RYN2}).

\medskip
The {\em constant coefficient} case
(when $a_\pm$ are constant)
has been extensively investigated.
In this case equation \eqref{heval.eq} can be written as
\begin{equation}\label{heval2.eq}
-\D_p(u) =
\al_+ \phi_p(u^+) - \al_- \phi_p(u^-) ,
\quad u \in D_p ,
\end{equation}
with $\alpha_\pm = a_\pm + \la$,
and the solvability conditions can be equivalently formulated in
terms of the \fuc spectrum, defined by
$$
\Si_F=\{(\al_+,\al_-)\in\R^2:
\text{\eqref{heval2.eq} has a non-trivial solution $u$}\}.
$$
Clearly, $\la\in\Si_H \Leftrightarrow (\al_+,\al_-)\in\Si_F$,
and the problem \eqref{heval2.eq} is
non-resonant if and only if $(\al_+,\al_-) \not\in \Si_F$.
The set $\Si_F$ is known explicitly
(see for instance \cite{DRA}).
Geometrically, it can be described as the union
$\Si_F = \cup_{k=0}^\infty \Si_k$,
where each set $\Si_k$, $k \ge 1$, consists of a pair of hyperbolic
curves
($\Si_0$ consists of a horizontal and a vertical line),
and the eigenfunctions corresponding to a point
$(\al_+,\al_-) \in \Si_k$
belong to the set $S_{k,+} \cup S_{k,-}$.
In the non-resonant, constant coefficient case, solvability conditions
for \eqref{nlslvble.eq} can be formulated in terms of the location
of the point $(\al_+,\al_-) \in \R^2$ relative to $\Si_F$ ---
see for example \cite{DAN,DRA}, and references therein.
In this case the half-eigenvalue and \fuc spectrum conditions are
equivalent.
In the variable  coefficient case, the $\Si_H$ solvability conditions
are more general than the $\Si_F$ conditions,
see \cite{RYN2}.

When $\la \in\Si_H$ (or, in the constant coefficient case,
$(\al_+,\al_-)\in\Si_F$), the problem is said to be  {\em resonant},
and is more delicate than the non-resonant case.
In particular, further (Landesman-Lazer) conditions on $f$, and
its interaction with the half-eigenfunctions, are required to obtain
solutions.

\medskip

The proof of our main result will
be based on a shooting method, and will make extensive use of the
properties of the
solutions $\Psi_{\la,\pm} \in D_p$ of the following initial value problems:
\begin{gather}
-\phi_p(\Psi_{\la,\pm}')' - a_+ \phi_p(\Psi_{\la,\pm}^+)
+ a_- \phi_p(\Psi_{\la,\pm}^-)
  = \la \phi_p(\Psi_{\la,\pm}) \quad\text{on} \ (0,1),
  \label{Phi_la_pm.eq}
\\
\Psi_{\la,\pm}(0) = 0 , \quad  \Psi_{\la,\pm}'(0) = \pm 1 .
  \label{Phi_la_pm_bc.eq}
\end{gather}

\begin{remark}\label{wellposedness.rem}
For all $\la \in \R$, problems \eqref{Phi_la_pm.eq}-\eqref{Phi_la_pm_bc.eq}
have unique solutions $\Psi_{\la,\pm} \in D_p$
(see, for example, \cite[Theorem~5]{RW2}).
In addition, by uniqueness, each $\Psi_{\la,\pm}$ has only simple zeros,
and hence $\Psi_{\la,\pm} \in S_{k,\pm}$, for some $k \ge 0$.
\end{remark}

In the course of the proof we will need certain technical results about
uniqueness and dependence on the parameters for some perturbations of the
problems
\eqref{Phi_la_pm.eq}-\eqref{Phi_la_pm_bc.eq}
(namely \eqref{psi_pm2.eq}-\eqref{psi_pm_bc2.eq}).
Due to the degeneracy of \eqref{Phi_la_pm.eq}
(resp. \eqref{psi_pm2.eq}) when the derivative of the solution vanishes,
these results do not fall within the scope of standard ODE theory,
and their proof will require the following technical assumption
on the coefficients $a_\pm$ (see Remark~\ref{wellpos.rem}).

\medskip\noindent
Let $\la \in \R$ be fixed.

\bbl
\item[$(C_\la)$]
If $p > 2$ then,
for each $\nu\in\{\pm\}$ and $x\in[0,1]$
for which $\Psi'_{\la,\nu}(x) = 0$,
$$
a_\pm(x) \ne 0 \quad \text{if} \quad \pm \Psi_{\la,\nu}(x)>0
$$
(by Remark~\ref{wellposedness.rem}, $\Psi_{\la,\nu}(x) \neq 0$ whenever
$\Psi'_{\la,\nu}(x) = 0$).
\ebl

\noindent
In particular, condition $(C_\la)$ is trivially satisfied if, for instance,
$a_\pm$ are positive constants --- this is the case of interest when
considering the \fuc spectrum.
\medskip

\medskip
We can now state our main result.

\begin{thm}  \label{landesman.thm}
Suppose that $f$ satisfies conditions $(f1)$-$(f2)$,
and $\la \in \{ \la_{k,\pm} \}$ for some $k \ge 0$.
If $p > 2$, suppose also that condition $(C_\la)$ holds.
\item[(A)]
If $\la_{k,\min} = \la_{k,\max}$ and
\begin{equation}\label{u_k_min_u_k_max.eq}
\left( \int_0^1  f_{+} u_{k,\min}^+ - f_{-} u_{k,\min}^-  \right)
\cdot
\left( \int_0^1  f_{+} u_{k,\max}^+ - f_{-} u_{k,\max}^-  \right)
> 0 ,
\end{equation}
then  \eqref{nlslvble.eq} has a solution.
\item[(B)]
If $\la_{k,\min} < \la_{k,\max}$ and
\begin{align}
&(\mathrm{B}1)
\quad\text{if} \ \la=\la_{k,\min},\qquad
\int_0^1 \big(f_{+} u_{k,\min}^+ - f_{-} u_{k,\min}^-  \big) < 0 ;
\label{u_k_min.eq} \\
&(\mathrm{B}2)
\quad \text{if} \ \la=\la_{k,\max}, \qquad
\int_0^1 \big(f_{+} u_{k,\max}^+ - f_{-} u_{k,\max}^-  \big) > 0 .
 \label{u_k_max.eq}
\end{align}
then  \eqref{nlslvble.eq} has a solution.
\end{thm}

\subsection{Previous results}  \label{prev_results.sec}

Existence conditions similar to those of Theorem~\ref{landesman.thm}
have been used for problems at resonance in many different contexts.
They seem to have first occurred in a paper by
Landesman and Lazer \cite{LL},
and are thus known as {\em Landesman-Lazer conditions},
but many papers since then have considered such problems.
Much of this literature has been concerned with asymptotically linear
(or, so-called, asymptotically `$p$-linear') problems,
that is, problems with $a_+=a_-$ so that the nonlinearity is not jumping,
and with resonance at an eigenvalue of the $p$-Laplacian,
rather than at a half-eigenvalue.
For brevity, we will not describe such results here.

However, various papers have also obtained existence results for resonant
problems with jumping nonlinearities of the form \eqref{nlslvble.eq},
see, for example, \cite{ASA,DONG,DR,FAB,PER,TOM}.
Most of these works deal with problems with constant coefficients
(that is, with $a_\pm$ constant)
and resonant with respect to the \fuc spectrum,
and give existence results based on Landesman-Lazer
conditions.

The papers \cite{DR,PER,TOM} consider \eqref{nlslvble.eq} in the constant
coefficient case, using a variational approach.
The paper \cite{TOM} considers the case $p=2$, whereas
\cite{DR,PER} deal with the general, quasilinear case, $p>1$
(\cite{PER} considers a partial differential equation problem).
We use a shooting method to study \eqref{nlslvble.eq}
with variable coefficients $a_\pm \in C^0[0,1]$, in the general
case $p>1$.
Except for some extra regularity assumptions on $f$
(which are required to ensure that certain initial value problems have
good solution properties),
we obtain more general results than those of \cite{DR,PER,TOM}, even in the
case of constant $a_\pm$.
In fact, \cite{DR} deals with the simpler case where $f$ does not depend
on $u$, and assumes that $(\al_+,\al_-) \in \Si_k$, with $k=0,1$ or 2.
On the other hand, \cite{PER,TOM} deal with a general nonlinearity $f$,
but require that
$$(\al_+,\al_-) \in \Si_k \cap (\la_{k-1},\la_{k+1})^2 \subset \R^2,$$
where $\la_k$, $k=0,1,\dots,$ are the eigenvalues of the $p$-Laplacian
(that is, $(\al_+,\al_-)$ lies on the portion of the \fuc spectrum
contained in the square whose vertices are at the points
$(\la_{k-1},\la_{k-1})$, $(\la_{k+1},\la_{k+1})$).
In addition, \cite{TOM} requires that $k$ be even ---
in the half-eigenvalue setting this corresponds to case (A) above.
In contrast to the results in these papers, our results hold for any point
$(\al_+,\al_-) \in \Si$
(and, more generally, in the variable coefficient case).

The papers \cite{ASA,DONG,FAB} consider the semilinear ($p=2$) case
only, with respectively Dirichlet, periodic, and general Sturm-Liouville
boundary conditions.
They all use an approach based on topological degree theory.
The papers \cite{ASA,FAB} consider the constant coefficient case,
while Dong \cite{DONG} considers the variable coefficient case.
His definition of resonance is expressed in terms of
`generalized resonant sets' that were defined and studied in
\cite{DONG0} --- these are essentially a non-constant version of the
\fuc spectrum, where the analogues of the sets $\Si_k$ depend on
$x\in[0,1]$.
Our result is somewhat similar to those of
\cite{ASA,DONG,FAB},
but we consider the general case, $p>1$.
However, in the case $p=2$ the Landesman-Lazer conditions of
\cite{ASA,DONG,FAB} do not require such a strict limiting behaviour of
$f$ as we assume in $(f1)$.
The main structural differences with our conditions can be briefly
described as follows.
In the \fuc spectrum context outlined above, let
$$
g(u) = \al_+ u^+ - \al_- u^- + f(u).
$$
The existence results of \cite{ASA,DONG,FAB} are based on conditions
similar to \eqref{u_k_min_u_k_max.eq}-\eqref{u_k_max.eq},
but these are expressed in terms of the `$\liminf$' and `$\limsup$' of $g$
as $u\to\pm\infty$.
In addition, they require that for $|u|$ sufficiently large,
the ratio $g(u)/u$ lies in a rectangular `box' whose vertices lie on
consecutive curves of the \fuc spectrum
(or in the generalized resonant sets in \cite{DONG}).
Instead, we use the limits in \eqref{fpm.eq} for our existence
conditions --- implying that $g(u)/u$ converges to one of these
vertices as $|u|\to\infty$
(more precisely $g(u)/u \to \al_\pm$ as $u\to \pm\infty$,
where $(\al_+,\al_-)\in\Si_F$) --- but we allow for oscillations outside
the box as we approach the limit.

We were unable to apply topological arguments similar to those of
\cite{ASA,DONG,FAB} to the general case $p > 1$, due to considerable
difficulties related to various integration by parts arguments,
which seem to fail when $p \ne 2$.

\section{Preliminary results }  \label{prelims.sec}

\subsection{Half-eigenvalues}  \label{hevals.sec}

We first describe some further preliminary results regarding the
half-eigenvalues.
By definition, for each $k \ge 0$ and $\nu \in \{\pm\}$,
\begin{equation}   \label{la_in_Si_iff.eq}
\text{$\la = \la_{k,\nu} \iff \Psi_{\la,\nu}(1)  = 0$
and 1 is the $(k+1)$th zero of $\Psi_{\la,\nu}$ in $(0,1]$}
% \iff \Psi_{\la,\nu}=u_{k,\nu} .
\end{equation}
(recall that $\Psi_{\la,\pm} \in D_p$ are the solutions of the
initial value problems \eqref{Phi_la_pm.eq}-\eqref{Phi_la_pm_bc.eq}
introduced in Section~\ref{intro.sec}).
Of course, if $\la = \la_{k,\nu}$ then $\Psi_{\la,\nu}=u_{k,\nu}$.
In particular, $\la_{k,+} = \la_{k,-}$, for some $k \ge 0$,
iff $\Psi_{\la,+}(1)  = \Psi_{\la,-}(1)  = 0$.
We will also require the following results.

\begin{lemma}   \label{Psi_Psid_sign.lem}
If,
for some $k \ge 0$ and $\nu \in \{\pm\}$,
$\la = \la_{k,\nu} = \la_{k,\min} < \la_{k,\max}$
then
\begin{equation}  \label{Psi_Psid_sign.eq}
\Psi'_{\la,\nu}(1) \,\, \Psi_{\la,-\nu}(1) > 0 .
\end{equation}
If $\la_{k,\min} <\la_{k,\max} =  \la = \la_{k,\nu}$ then the inequality
in \eqref{Psi_Psid_sign.eq} is reversed.
\end{lemma}

\begin{proof}
Suppose that $\la = \la_{k,\nu} = \la_{k,\min} < \la_{k,\max}$.
Then it follows from Theorem~\ref{heval.thm} and
\eqref{la_in_Si_iff.eq}, together with standard Sturmian-type
theory for the solutions of
\eqref{Phi_la_pm.eq},  \eqref{Phi_la_pm_bc.eq},
that each of the functions $\Psi_{\la,\pm}$ has exactly $k$ zeros in
the interval $(0,1)$ (and their derivatives are non-zero at these zeros,
so they change sign at them).
Hence, it follows from the boundary conditions \eqref{dbc.eq}
that $\Psi_{\la,\pm}$ are non-zero and have opposite sign on some
interval $(1-\de,1)$, for sufficiently small $\de > 0$,
which proves \eqref{Psi_Psid_sign.eq}.
The proof of the other case is similar.
\end{proof}

The following result can be obtained readily from
Lemma~\ref{Psi_Psid_sign.lem} and its proof.

\begin{cor}   \label{la_in_Si_gap_iff.cor}
\begin{equation}   \label{la_in_Si_gap_iff.eq}
% \text{\rm $\la_{k,\min} < \la < \la_{k,\max}$ \ for some $k \ge 0$}
\la_{k,\min} < \la < \la_{k,\max}
\iff
\Psi_{\la,+}(1) \Psi_{\la,-}(1) > 0 .
\end{equation}
\end{cor}

\begin{remark}
In the case $p=2$ the result of Corollary~\ref{la_in_Si_gap_iff.cor}
was, in essence, proved in
\cite[Proposition 1]{DAN} and \cite[Theorem 5.5]{RYN1},
although it was stated somewhat differently, and was obtained in the
course of the proof of a different result
(\cite{DAN} considered the constant coefficient
\fuc case, while \cite{RYN1} considered the half-eigenvalue case with
coefficients $a_\pm \in L^\infty(0,1)$).
Lemma~\ref{Psi_Psid_sign.lem} was not needed in these papers so was not
considered there.
The general case $2 \ne p>1$ was not considered in either of these papers.
\end{remark}

\section{Proof of Theorem \ref{landesman.thm}}  \label{landesman.sec}

To simplify the notation slightly we observe that
by replacing $a_\pm$ by  $a_\pm + \la$,
we may suppose, without loss of generality, that $\la = 0$.
With this supposition we will now write
$\Psi_{\pm} := \Psi_{0,\pm}$.

For any $\tau \in \R$, let $\psi(\tau) \in D_p$ denote a
solution of the initial value problem
\begin{gather}
-\phi_p(\psi(\tau)')' = a_+ \phi_p(\psi(\tau)^+) - a_- \phi_p(
\psi(\tau)^-)
  +  f(\psi(\tau)) ,
  \label{tpsi_pm.eq}
\\
\psi(\tau)(0) = 0 ,
\quad  \psi(\tau)'(0) =  |\tau|^{\frac{1}{p-1}}\sgn\tau
\label{tpsi_pm_ic.eq}
\end{gather}
(in general, solutions of
\eqref{tpsi_pm.eq}-\eqref{tpsi_pm_ic.eq}
need not be unique).
By definition,
$\psi(\tau)$ satisfies \eqref{eq.eq}-\eqref{dbc.eq}
(with $\la=0$) if and only if
\begin{equation}  \label{psi_pm_zero.eq}
\psi(\tau)(1) = 0.
\end{equation}
We will prove Theorem~\ref{landesman.thm} by showing that there
exists $\tau \in \R$ and a corresponding solution $\psi(\tau)$ of
\eqref{tpsi_pm.eq}-\eqref{tpsi_pm_ic.eq}
for which \eqref{psi_pm_zero.eq} holds.
This will require several steps.

\bil
\item
Firstly, we will show that there exists
a `large' $\tau_0 > 0$ such that the problems
\eqref{tpsi_pm.eq}-\eqref{tpsi_pm_ic.eq},
with $\tau = \pm \tau_0$,
have unique solutions $\psi(\pm \tau_0)$,
and these satisfy
\beq  \label{psi_tau_large.eq}
\psi(\tau_0)(1) \, \psi(-\tau_0)(1) < 0 .
\eeq
To do this it will be convenient to first rescale the problem so
that a `large' $\tau$ corresponds to a `small' $\ttau$.
\item
If the values $\psi(\tau)(1)$ were unique and depended continuously on
$\tau \in \R$ then \eqref{psi_pm_zero.eq} would follow immediately from
\eqref{psi_tau_large.eq}.
Unfortunately, this deduction is not so simple since, for `small'
$\tau$, the solutions of the problems
\eqref{tpsi_pm.eq}-\eqref{tpsi_pm_ic.eq}
need not be unique.
However, we will use a connectedness property of the set of solutions of
\eqref{tpsi_pm.eq}-\eqref{tpsi_pm_ic.eq}
to obtain \eqref{psi_pm_zero.eq} from
\eqref{psi_tau_large.eq}.
\eil
We will carry out these steps in the following subsections.

\subsection{A rescaled problem}\label{rescale.sec}

For any $\ttau \in \R$, we consider the initial value problems
\begin{gather}
-\phi_p(u')' = a_+ \phi_p(u^+) - a_- \phi_p(u^-) +
\ttau f\big(|\ttau|^{-\frac{1}{p-1}}u\big) , 	\label{psi_pm2.eq}
\\
u(0) = 0 , \quad  u'(0) = \pm 1 ,		\label{psi_pm_bc2.eq}
\end{gather}
where we regard the final term on the right hand side of
\eqref{psi_pm2.eq}
as 0 when $\ttau = 0$
(this is consistent with  \eqref{fint.eq}).
It then follows from Remark~\ref{wellposedness.rem} that,
when $\ttau = 0$, these problems
have the unique solutions $\Psi_{\pm} \in D_p$,
defined on the whole interval $[0,1]$.
%(see part~$(\ga)$ of \cite[Theorem~4]{RW1}).
However, in general, solutions of
\eqref{tpsi_pm.eq}-\eqref{tpsi_pm_ic.eq}
need not be unique (see \cite{RW1}).

\begin{remark}
Comparing the problems
\eqref{tpsi_pm.eq}-\eqref{tpsi_pm_ic.eq}
and
\eqref{psi_pm2.eq}-\eqref{psi_pm_bc2.eq}
shows that
\begin{equation}  \label{psi_tau_scaling.eq}
\psi(\tau) = |\tau|^{\frac{1}{p-1}} \tpsi_{\sgn \tau} (|\tau|^{-1}),
 \quad \tau \ne 0.
% \psi(\ttau^{-1})
%    := |\ttau|^{-\frac{1}{p-1}} \tpsi_{\sgn \ttau} (|\ttau|) ,
\end{equation}
(Given the non-uniqueness of the solutions of these problems in
general,
we mean by this that if one side of \eqref{psi_tau_scaling.eq} is a
solution then the other side is a solution, of their
respective problems.)
Thus, we will show that \eqref{psi_tau_large.eq} holds
for some `large' $\tau_0$
by proving that $\tpsi_+(\ttau_0)(1) \tpsi_-(\ttau_0)(1) < 0$
for some `small' $\ttau_0 > 0$
(see Lemma~\ref{sign_tau_small.lem} and
Proposition~\ref{sign_tau_large.prop} below).
\end{remark}

\begin{remark}\label{wellpos.rem}
Despite the general lack of uniqueness of solutions of
\eqref{psi_pm2.eq}-\eqref{psi_pm_bc2.eq},
we will show that uniqueness holds when $|\ttau|$ is sufficiently small.
Our proof of this in the case $p>2$ will use condition $(C_\la)$,
which implies that
\begin{equation}\label{RHScondition}
- a_+(x) \phi_p(\Psi_\nu^+(x)) + a_-(x)\phi_p(\Psi_\nu^-(x))  \neq 0
\quad\text{whenever}\quad \Psi'_\nu(x) = 0 ,
\end{equation}
and this will ensure that solutions of the initial value
problems
\eqref{psi_pm2.eq}-\eqref{psi_pm_bc2.eq}
(for sufficiently small $|\ttau|$)
do not behave badly near points where their derivatives vanish.
% It will allow us to prove the following
% proposition about uniqueness of solutions, and will also play a role in
% the proof of Proposition~\ref{deriv_psi0.prop} below, about the
% differentiability
% of solutions at $\ttau=0$.
\end{remark}

\begin{prop} \label{ivp_props_tpsi.prop}
There exists $\tde > 0$ such that if $|\ttau| < \tde$ then$:$
\bal
\item
\eqref{psi_pm2.eq}-\eqref{psi_pm_bc2.eq}
has a unique solution $\tpsi_{\pm}(\ttau) \in D_p;$
\item
the mappings
$\ttau \to \tpsi_{\pm}(\ttau) : (-\tde,\tde) \to C^1[0,1]$
are continuous.
In particular,
\beq  \label{psi_eq_Phiz.eq}
\tpsi_{\pm}(0) = \Psi_{\pm} .
\eeq
\eal
\end{prop}

The proof of Proposition~\ref{ivp_props_tpsi.prop} will be postponed
to Appendix~\ref{app_uni.sec} to avoid disrupting the main argument
here.

\medskip
We will now show that the mappings
$\ttau \to \tpsi_{\pm}(\ttau) : (-\tde,\tde) \to C^1[0,1]$
found in Proposition~\ref{ivp_props_tpsi.prop}
are differentiable with respect to $\ttau$, at $\ttau = 0$.
These derivatives will be denoted by $\tpsi_{\pm,\ttau}^0$.
For any $u\in C^1[0,1]$, we  define the functions
$$
f_{\sgn u}(x) :=
\begin{cases}
f_{\sgn (u(x))}(x),  & \text{if $u(x) \ne 0$,}
\\
0 ,  & \text{if $u(x) = 0$}
\end{cases}
$$
(recall that the functions $f_\pm$ were defined in \eqref{fpm.eq}),
and we let $\chi_{u}^\pm$ denote the characteristic function of the
set $\{x \in [0,1] : u^\pm (x) > 0\}$.

\begin{prop}\label{deriv_psi0.prop}
For each $\nu \in \{\pm\}$,
the derivative $\tpsi_{\nu,\ttau}^0$ exists in $C^1[0,1]$
and satisfies the linear initial value problem
\begin{gather}
-(p-1) \left[ \big( |\Psi_\nu'|^{p-2}(\tpsi_{\nu,\ttau}^0)'\big)'
 + \big(a_+ |\Psi_\nu^+|^{p-2} \chi_{\Psi_\nu}^+
 - a_- |\Psi_\nu^-|^{p-2} \chi_{\Psi_\nu}^- \big) \tpsi_{\nu,\ttau}^0
\right]
  = f_{\sgn \Psi_\nu}, \label{psi_pm_dtau.eq}
\\
\tpsi_{\nu,\ttau}^0(0) = 0 , \quad  (\tpsi_{\nu,\ttau}^0)'(0) = 0 .
  \label{psi_pm_dtau_bc.eq}
\end{gather}
Furthermore, the coefficients in
\eqref{psi_pm_dtau.eq} satisfy the standard assumptions
on linear Sturm-Liouville problems, that is,
$$
1/|\Psi_\nu'|^{p-2} \in L^1(0,1) \quad  \text{and}  \quad
|\Psi_\nu^\pm|^{p-2} \chi_{\Psi_\nu}^\pm  \in L^1(0,1).
$$
\end{prop}

The proof of Proposition~\ref{deriv_psi0.prop} is somewhat lengthy, so
we give it in Appendix~\ref{app_diff.sec}.

\medskip

Using the above differentiability results we can now prove the
following lemma.

\begin{lemma} \label{sign_tau_small.lem}
If $\ttau > 0$ is sufficiently small then
\begin{equation}  \label{tpsi_pm_neg.eq}
\tpsi_+(\ttau)(1) \,\, \tpsi_-(\ttau)(1)  < 0.
\end{equation}
\end{lemma}

\begin{proof}
{\bf Case (A).}
In this case we  have  $\la=\la_{k,+}=\la_{k,-}=0$,
and $u_{k,\pm} = \Psi_{\pm}$, so
\begin{gather}
-\D_p(u_{k,\pm}) - a_+ \phi_p(u_{k,\pm}^+) + a_- \phi_p(u_{k,\pm}^-)  =
0,
  \label{u_pm.eq}
\\
u_{k,\pm}(0) = 0, \quad  u_{k,\pm}'(0) = \pm 1 , \quad  u_{k,\pm}(1) = 0
.
  \label{u_pm_bc.eq}
\end{gather}
By \eqref{psi_eq_Phiz.eq} and the final condition in \eqref{u_pm_bc.eq}
we have $\tpsi_{\pm}(0)(1)=0$,
so we need only show that
\begin{equation}  \label{tpsi_ttau_pm_neg.eq}
\tpsi_{-,\ttau}^0(1) \, \tpsi_{+,\ttau}^0(1)<0.
\end{equation}
Multiplying \eqref{psi_pm_dtau.eq} by $u_{k,\pm}$,
integrating by parts, and using
\eqref{psi_pm_dtau_bc.eq}-\eqref{u_pm_bc.eq}
yields
$$
(p-1)|u_{k,\pm}'(1)|^{p-2} u_{k,\pm}'(1) \tpsi_{\pm,\ttau}^0(1)
=
\int_0^1  f_{\sgn u_{k,\pm}} u_{k,\pm}
$$
\begin{equation}\label{bterms.eq}
= \int_0^1 \big( f_{+} u_{k,\pm}^+ - f_{-} u_{k,\pm}^-  \big)  .
\end{equation}
It now follows from \eqref{u_k_min_u_k_max.eq} and \eqref{bterms.eq}
that
$$
u_{k,-}'(1) u_{k,+}'(1) \tpsi_{-,\ttau}^0(1) \tpsi_{+,\ttau}^0(1) > 0,
$$
and hence  \eqref{tpsi_ttau_pm_neg.eq} holds, which completes
the proof of  Lemma~\ref{sign_tau_small.lem} in case (A).

\medskip

\noindent
{\bf Case (B).} We will deal with the case (B1), assuming that
\begin{equation}  \label{B1_supp.eq}
\la_{k,\min} = \la_{k,+} = 0 < \la_{k,-},
\end{equation}
so that, by definition, $u_{k,\min} = u_{k,+} = \Psi_{+}$.
The other cases are similar.
We now observe that, by \eqref{psi_eq_Phiz.eq},
$\tpsi_{-}(0)(1) = \Psi_{-}(1) \ne 0$
so by continuity, for $\ttau > 0$ sufficiently small,
\begin{equation} \label{psim_Phiz.eq}
\tpsi_{-}(\ttau)(1) \, \Psi_{-}(1)  > 0 ,
\end{equation}
that is, $\tpsi_{-}(\ttau)(1)$ has the same sign as $\Psi_{-}(1)$.
On the other hand, by \eqref{la_in_Si_iff.eq} and
\eqref{psi_eq_Phiz.eq},
$\tpsi_{+}(0)(1) =  0$,
so to obtain a similar result to \eqref{psim_Phiz.eq}
for the sign of $\tpsi_{+}(\ttau)(1)$, for $\ttau$ small,
we will consider the sign of the derivative
$\tpsi_{+,\ttau}^0$.

In this case the argument in the proof of case (A) now shows that
$$
(p-1)|u_{k,\min}'(1)|^{p-2} u_{k,\min}'(1) \tpsi_{+,\ttau}^0(1)
= \int_0^1 f_{\sgn u_{k,\min}}  u_{k,\min}
$$
\begin{equation}\label{bterms2.eq}
= \int_0^1 \big( f_{+} u_{k,\min}^+ -
 f_{-} u_{k,\min}^-  \big) < 0 ,
\end{equation}
by \eqref{u_k_min.eq},
that is, $\tpsi_{+,\ttau}^0(1) u_{k,\min}'(1) < 0$.
Also, it follows from  \eqref{B1_supp.eq} and
Lemma~\ref{Psi_Psid_sign.lem} that
\begin{equation}  \label{u_min_phi.eq}
u_{k,\min}'(1) \Psi_{-}(1)  = \Psi_{+}'(1) \Psi_{-}(1) > 0 .
\end{equation}
Hence, by \eqref{bterms2.eq} and \eqref{u_min_phi.eq},
$\tpsi_{+,\ttau}^0(1) \, \Psi_{-}(1)  < 0 ,$
and so, for $\ttau > 0$ sufficiently small,
\begin{equation}  \label{psip_Phiz.eq}
\tpsi_{+}(\ttau)(1) \, \Psi_{-}(1)  < 0 .
\end{equation}
Combining \eqref{psim_Phiz.eq} and \eqref{psip_Phiz.eq} completes the
proof of Lemma~\ref{sign_tau_small.lem}.
\end{proof}

Lemma~\ref{sign_tau_small.lem}, together with
\eqref{psi_tau_scaling.eq}, now yields the the following result.

\begin{prop} \label{sign_tau_large.prop}
There exists  $\tau_0 > 0$ such that the initial value
problems
\eqref{tpsi_pm.eq}-\eqref{tpsi_pm_ic.eq},
with $\tau = \pm \tau_0$,
have unique solutions $\psi(\pm \tau_0)$,
satisfying \eqref{psi_tau_large.eq}.
\end{prop}

% \begin{proof}
% Proposition~\ref{sign_tau_large.prop} now follows immediately
% from \eqref{psi_tau_scaling.eq} and Lemma~\ref{sign_tau_small.lem}.
% \end{proof}

\subsection{Connectedness of the set of solution values}

If the initial value problem
\eqref{tpsi_pm.eq}-\eqref{tpsi_pm_ic.eq}
had a unique solution $\psi(\tau)$ for all $\tau \in [-\tau_0,\tau_0]$,
with $\psi(\tau)(1)$ a continuous function of $\tau$,
then we could immediately conclude from
Proposition~\ref{sign_tau_large.prop}
that \eqref{psi_pm_zero.eq} has a solution,
which would then complete the proof of Theorem~\ref{landesman.thm}.
Unfortunately, uniqueness of solutions of
\eqref{tpsi_pm.eq}-\eqref{tpsi_pm_ic.eq} may fail
for `small' $\tau$, so we need to deal with this possibility.

For any $\tau \in \R$, let $S_\tau$ denote the set of all
solutions $\psi \in D_p$ of
\eqref{tpsi_pm.eq}-\eqref{tpsi_pm_ic.eq}
(we note that, by the proof of Proposition~\ref{ivp_props_tpsi.prop},
any local solution of
\eqref{tpsi_pm.eq}-\eqref{tpsi_pm_ic.eq}
extends to the interval $[0,1]$).

\begin{prop} \label{connectedness.prop}
The set of solution values at $x=1$,
$$
V(1) :=
\bigcup_{\tau \in [-\tau_0,\tau_0]} \bigcup_{\psi \in S_{\tau}}
\psi(1),
$$
is connected.
\end{prop}

\begin{proof}
If $V(1)$ is not connected then there exist compact, non-empty sets
$M_1$, $M_2$ such that
$$
V(1) = M_1 \cup M_2, \quad M_1 \cap M_2 = \emptyset.
$$
Adapting the proof of
\cite[Proposition~13.9]{ZEI} (which deals with a similar connectedness
result for a first order initial value problem),
one can show that, for each fixed $\tau \in [-\tau_0,\tau_0]$, the set
$$
V_\tau(1) := \bigcup_{\psi \in S_{\tau}} \psi(1)
$$
is connected.
Hence, $V_\tau(1) \subset M_i$, for some $i = 1,\,2$,
and we can define the non-empty sets
$$
T_i := \{ \tau \in [0,1]: V_\tau(1) \subset M_i \} ,
\quad i = 1,\,2 .
$$
Clearly, there exists a point $\tau^* \in [0,1]$ which is a limit point
of each of the sets $T_1,\,T_2$.
However, by a similar argument to the proof of
Proposition~\ref{ivp_props_tpsi.prop}~$(b)$, if $(\tau_n)$ is
a sequence with $\tau_n \to \tau^*$, and $(\psi_{\tau_n}(1))$ is a
corresponding sequence of solution values then
(after taking a subsequence if necessary)
$\psi_{\tau_n}(1) \to \psi_{\tau^*}(1)$
for some $\psi_{\tau^*} \in S_{\tau^*}$.
We deduce that  $V_{\tau^*}(1) \cap M_i \ne \emptyset$,
$i = 1,\,2$,
which contradicts the connectedness of the set $V_{\tau^*}(1)$.
\end{proof}

\subsection{Conclusion of the proof of Theorem~\ref{landesman.thm}}
It now follows immediately from
Propositions~\ref{sign_tau_large.prop}
and~\ref{connectedness.prop}
that equation \eqref{psi_pm_zero.eq} has a solution $\tau$,
which completes the proof of Theorem~\ref{landesman.thm}.
\hfill$\Box$

\appendix

\section{Proof of
Proposition~\ref{ivp_props_tpsi.prop}}\label{app_uni.sec}

\noindent $(a)$
By \cite[Theorem~1]{RW1}, for any $\ttau \in \R$ the problem
\eqref{psi_pm2.eq}-\eqref{psi_pm_bc2.eq}
has a local solution $u$ (possibly non-unique) on an interval
$I_u \subset [0,1]$ containing $0$
(for now, we will suppose that one or other of the $\pm$ signs in
\eqref{tpsi_pm_ic.eq} has been chosen and is fixed,
so we omit these, and the dependence of $u$ on $\ttau$,
from the notation).
We now define
$$
m(x):=\max\{|u(y)|:0\le y \le x\},
\quad x \in I_u .
$$
By integrating \eqref{psi_pm2.eq}, and
using \eqref{psi_pm_bc2.eq}, we see that
$$
|\phi_p(u'(x))|\le 1 + K_0 |\ttau| + m(x)^{p-1}\int_0^x(|a_+|+|a_-|),
\quad x \in I_u ,
$$
so that
\beq  \label{ud_bnd.eq}
|u'(x)| \le C( M_\ttau + m(x) ),   \quad x \in I_u,
\eeq
where $M_\ttau := 1+|\ttau|^{\frac{1}{p-1}}$
(here, and below, $C$ will denote a positive constant, which may be
different on each occasion but which does not depend on $\ttau$ of $u$).
Now, by a further integration and \eqref{ud_bnd.eq},
$$
m(x) \le C M_\ttau + C \int_0^x m  , \quad x \in I_u,
$$
so by Gronwall's inequality and \eqref{ud_bnd.eq},
\beq  \label{Gronwall_est.eq}
|u(x)| + |u'(x)| \le  C M_\ttau ,  \quad x \in I_u .
\eeq
% \beq  \label{Gronwall_est.eq}
% |u(x)| \le m(x)\le C M \mathrm{e}^{Cx}  , \quad x \in I_u .
% \eeq
% It follows that
% \beq\label{Gronwall_est2.eq}
% |u'(x)| \le CM_\ttau(1 + \mathrm{e}^{Cx})  , \quad x \in I_u,
% \eeq
The Corollary to \cite[Theorem~1]{RW1} now shows that the solution $u$
may
be extended to the whole of the interval $[0,1]$
(a priori, not necessarily in a unique manner).
Hence, from now on we may suppose that any solution $u$ of
\eqref{psi_pm2.eq}-\eqref{psi_pm_bc2.eq}
is defined on $[0,1]$ and belongs to $D_p$.

Next, we show that we can choose a  sufficiently small $\tde>0$ such
that if $|\ttau| \le \tde$ then the solution $u$ is locally unique in a
neighbourhood of any point $x_0 \in[0,1]$.
Suppose that $u(x_0)=u_0, \ u'(x_0)=u'_0.$
By integrating from the point $x=x_0$, and using a similar argument to
the
proof of \eqref{Gronwall_est.eq},
we can show that $u$ must satisfy
$$
|u|_1 \le  C \big(|u_0|+|u'_0|+|\ttau|^{\frac{1}{p-1}} \big)  .
$$
Now, since $|u'(0)|=1$ (by \eqref{psi_pm_bc2.eq}),
it follows immediately that if $\tde$ is sufficiently small and
$|\ttau| \le \tde$ then $|u_0|+|u'_0|>0$,
that is,
\beq  %\label{u_non_zero.eq}
|u(x_0)| + |u'(x_0)| > 0 , \quad x_0 \in [0,1] .
\eeq
In particular, $u'$ is non-zero at any zero of $u$, and vice versa.
It now follows from our hypotheses on $a_\pm$ and $f$,
together with parts~($\al$)-($ii$), ($\be$)-($iii$) and ($\be$)-($v$)
of \cite[Theorem~4]{RW1}, that the solution $u$ is locally unique near
$x_0$.
Since $x_0$ was arbitrary, this yields global uniqueness on $[0,1]$.
Having proved existence and uniqueness of solutions on $[0,1]$, when
$|\ttau| \le \tde$,
we can now use the notation
$\tpsi_\pm(\ttau) \in D_p$ for these solutions.
\medskip

\noindent $(b)$
We now fix $\nu \in \{\pm\}$ and $\ttau_\infty \in [-\tde,\tde]$,
and consider a sequence $(\ttau_n)$ in $[-\tde,\tde]$ such that
$\ttau_n \to \ttau_\infty$.
% and let $(u_{n,\nu})$ be a corresponding sequence of solutions of
% \eqref{psi_pm2.eq}-\eqref{psi_pm_bc2.eq} in $D_p$.
It follows from \eqref{Gronwall_est.eq}
and the compactness of the
embedding of $C^1[0,1]$ into  $C^0[0,1]$, that
$\tpsi_\nu(\ttau_n) \to u_{\infty,\nu}$ in $C^0[0,1]$,
for some $u_{\infty,\nu} \in C^0[0,1]$
(after taking a subsequence if necessary).
It then follows from the integral form of
\eqref{psi_pm2.eq}-\eqref{psi_pm_bc2.eq}
that $\tpsi_\nu(\ttau_n) \to u_{\infty,\nu}$ in $C^1[0,1]$,
and that $u_{\infty,\nu}$ is a solution of
\eqref{psi_pm2.eq}-\eqref{psi_pm_bc2.eq}
with $\ttau=\ttau_\infty$.
Hence, by the uniqueness proved in part~$(a)$,
$u_{\infty,\nu} = \tpsi_\nu(\ttau_\infty)$.
We deduce that the mapping $\ttau \to \tpsi_\nu(\ttau)$
is continuous at $\ttau_\infty$.

\section{Proof of Proposition~\ref{deriv_psi0.prop}}
\label{app_diff.sec}

Due to the degeneracy of \eqref{psi_pm2.eq} when either $u=0$ or $u'=0$
(depending on the value of $p$),
standard results regarding the differentiability of solutions
with respect to parameters do not apply immediately.
However, we will essentially follow the proof of
\cite[Theorem~3.1,~p.~95]{HAR}
to establish our result.

We first rewrite the problem
\eqref{psi_pm2.eq}-\eqref{psi_pm_bc2.eq}
as an initial value problem for a first order system, in which
we incorporate the parameter $\ttau$ into the initial data.
Defining $v:=\phi_p(u')$, so that $u'=\phi_{p'}(v)$,
where $p' = p/(p-1)$,
we consider the initial value problem
\begin{equation}  \label{first_order_IVP.eq}
\mat{u \\ v \\ \ttau}' =
\mat{ \phi_{p'}(v) \\ g(u,\ttau) \\ 0} \quad \text{on $(0,1)$} ,
\quad
\mat{ u \\ v \\ \ttau}(0) =
\mat{0 \\ \nu \\ \ttau },
\end{equation}
where
$$
g(x,\xi,\ttau):=
-a_+(x) \phi_p(\xi^+) + a_-(x) \phi_p(\xi^-) -
  \ttau f\big( x,|\ttau|^{-\frac{1}{p-1}}\xi \big),
\quad x \in [0,1], \ \xi ,\, \ttau \in \R.
$$
As before, we regard the final term of $g$ as 0 when $\ttau = 0$,
and we now regard $\nu$ as an element of
$\{\pm 1\}$, rather than of $\{\pm \}$.
Clearly, the problems \eqref{psi_pm2.eq}-\eqref{psi_pm_bc2.eq} and
\eqref{first_order_IVP.eq} are equivalent so,
by Proposition~\ref{ivp_props_tpsi.prop},
when $|\ttau|<\tde$ the problem
\eqref{first_order_IVP.eq} has a unique solution
$\eta(\ttau) \in C^0[0,1]$
(for brevity, in this notation we omit the dependence of $\eta(\ttau)$
on $\nu$, which we regard as fixed for the remainder of the proof;
we also write $C^i[0,1] = C^i([0,1],\R^3)$, $i=0,1,$
which should not cause any confusion).
Proposition~\ref{ivp_props_tpsi.prop} also shows that the mapping
$\ttau \to \eta(\ttau) : (-\tde,\tde) \to C^0[0,1]$
is continuous;
we will now prove that this mapping is also differentiable
at $\ttau=0$,
and that the corresponding derivative,
$\eta_\ttau(0) $, satisfies $\eta_\ttau(0) = z_0$,
where $z_0$ is the (unique) solution of the linear initial
value problem
\begin{equation}  \label{linpb}
z' = J_0 z \quad \text{on $(0,1)$} , \quad
z(0)=
\mat{ 0 \\ 0 \\ 1 },
\end{equation}
with
\begin{equation*}
J_0 :=
\mat{
0 & \frac{1}{p-1}|\Psi_\nu'|^{2-p} & 0 \\
-a_+ |\Psi_\nu^+|^{p-2} \chi_{\Psi_\nu}^+
  +  a_- |\Psi_\nu^-|^{p-2} \chi_{\Psi_\nu}^- & 0 & - f_{\sgn \Psi_\nu}
\\
0 & 0 & 0
}.
\end{equation*}
(We will see in the proof of Lemma~\ref{J_L_cvgnce.lem} below that
each component of $J_0$ lies in $L^1(0,1)$,
so \eqref{linpb} indeed has a unique solution $z_0 \in C^0[0,1]$,
see \cite[Prob.~1,~Ch.~3]{CL}.)

For $0 < |\ttau|<\tde$ we define the difference quotient
$$
Q(\ttau) := \frac{1}{\ttau} (\eta(\ttau)-\eta(0))  \in C^0[0,1] .
$$
This satisfies an initial value problem which we now describe.
We begin by noting that, by the above construction, the first component of
$\eta(\ttau)$, which we denote by $\eta_1(\ttau)$,
in fact lies in $C^1[0,1]$, and depends continuously on $\ttau$ in
$C^1[0,1]$,
with
$\eta_1(0) = \Psi_\nu$.
Now, defining
$$
\ell(t,\ttau) := t \eta_1(\ttau) + (1-t)\Psi_\nu  \in C^1[0,1] ,
\quad t \in [0,1],
$$
and following the proof of \cite[Theorem~3.1,~p.~95]{HAR}
(using the integral form of the mean-value theorem),
we can show that $Q(\ttau)$ satisfies the initial value problem
\begin{equation}\label{zeq}
Q(\ttau)' = J(\ttau) Q(\ttau) \quad \text{on $(0,1)$}, \quad
Q(0) = \mat{ 0 \\ 0 \\ 1 } ,
\end{equation}
where the matrix $J(\ttau)$ has the following entries:
\alns{
J(\ttau)_{12}	&=  \frac{1}{p-1}\int_0^1
  |\ell(t,\ttau)'|^{2-p}  \diff t,
\\
J(\ttau)_{21}	&=  \int_0^1
   g_\xi(\ell(t,\ttau),\ttau)  \diff t ,
\\
J(\ttau)_{23}	&=  \int_0^1
   g_\ttau(\ell(t,\ttau),\ttau)  \diff t ,
}
and the other entries are zero.

\begin{lemma}  \label{J_L_cvgnce.lem}
\beq \label{J_lim.eq}
\lim_{\ttau \to 0} \|J(\ttau) - J_0 \|_1 = 0.
\eeq
\end{lemma}

\begin{remark}
In \eqref{J_lim.eq} we define the norm, $|A|_M$, of a matrix $A$
 to be the sum over the entries $\sum_{ij} |a_{ij}|$,
and we define the $L^1$ norm, $\|A\|_1$, of a matrix function $A(\cdot)$ to
be the $L^1$ norm of the matrix norm, viz.,
$\|A\|_1 :=\| |A(\cdot)|_M \|_1$.
\end{remark}

\begin{proof}
We consider the components of $J(\ttau)$ separately.

\noindent \underline{$J(\ttau)_{21}$} \
By definition,
\alns{
g_\xi(\ell(t,\ttau),\ttau)
=&
-a_+|\ell(t,\ttau)^+|^{p-2}\chi_{\ell(t,\ttau)}^+
+
a_-|\ell(t,\ttau)^-|^{p-2}\chi_{\ell(t,\ttau)}^-
\\&
-\ttau|\ttau|^{-\frac{1}{p-1}} f_\xi(|\ttau|^{-\frac{1}{p-1}}
   \ell(t,\ttau))
.
}
For each $t \in [0,1]$, the  mapping
$\ttau \to \ell(t,\ttau)$
is continuous into $C^1[0,1]$ at $\ttau = 0$,
with
$\eta_1(0) = \Psi_\nu$.
Hence, since the function $\Psi_\nu$ has only simple zeros in $[0,1]$,
it follows from \cite[Lemma~2.1]{BR} that the mappings
$\ttau \to |\ell(t,\ttau)^\pm|^{p-2}\chi_{\ell(t,\ttau)}^\pm$
are continuous into $L^1(0,1)$ at $\ttau = 0$,
with
$|\ell(t,0)^\pm|^{p-2}\chi_{\ell(t,0)}^\pm =
  |\Psi_\nu^\pm|^{p-2} \chi_{\Psi_\nu}^\pm$.
Also, we can write the final term on the RHS in the form
$$
-\ttau|\ttau|^{-\frac{1}{p-1}}
\big( |\ttau|^{-\frac{1}{p-1}} \ell(t,\ttau) \big)^{-\rho}
\big( |\ttau|^{-\frac{1}{p-1}} \ell(t,\ttau) \big)^{\rho}
f_\xi(|\ttau|^{-\frac{1}{p-1}} \ell(t,\ttau)) ,
$$
so it follows from assumption ($f3$) and \cite[Lemma~2.1]{BR}
that this term tends to zero in $L^1(0,1)$ as $\ttau\to0$.
Hence, by Fubini's theorem and dominated convergence,
$$\lim_{\ttau \to 0} \| J(\ttau)_{21} - (J_0)_{21}\|_1 = 0 .$$

\noindent \underline{$J(\ttau)_{12}$} \
By assumption ($C$),
$\Psi_\nu'(x)=0 \Rightarrow g(\Psi_\nu(x),0)  \neq 0$.
Also, for each $t \in [0,1]$, the  mapping
$\ttau \to g(t,\ttau)$
is continuous into $C^0[0,1]$ at $\ttau = 0$,
so it follows from
\eqref{psi_pm2.eq}-\eqref{psi_pm_bc2.eq} and \cite[Theorem~2.2]{RYN10}
that the mapping
$\ttau \to |\ell(t,\ttau)'|^{2-p}$
is continuous into $L^1(0,1)$ at $\ttau = 0$.
Hence, by Fubini's theorem and dominated convergence,
$$\lim_{\ttau \to 0} \| J(\ttau)_{12} - (J_0)_{12}\|_1 = 0 .$$

\noindent \underline{$J(\ttau)_{23}$} \
By definition,
\alns{
g_\ttau(\ell(t,\ttau),\ttau)
=&
f(|\ttau|^{-\frac{1}{p-1}}\ell(t,\ttau))
  +
\frac{1}{p-1}|\ttau|^{-\frac{1}{p-1}}\ell(t,\ttau)
  f_\xi(|\ttau|^{-\frac{1}{p-1}}\ell(t,\ttau))
.
}
Now, using $(f1)$-$(f2)$, a similar argument to the first case
shows that
$$
\lim_{\ttau \to 0} \| J(\ttau)_{23}  - f_{\sgn \Psi_\nu} \|_1 = 0 .
$$

Combining these results proves \eqref{J_lim.eq}.
\end{proof}

For $0 < |\ttau| < \tde$, we now define $m(\tau) \in C^0[0,1]$ by
$$
m(\ttau)(x) := \max\{|Q(\ttau)(y) - z_0(y)| : 0 \le y \le x\},
\quad  x \in [0,1] ,
$$
where $z_0$ denotes the unique solution of \eqref{linpb}.
Integrating \eqref{zeq} yields
\alns{
m(\ttau)(x)
&\le
|z_0|_0 \int_0^x |J(\ttau) - J_0|_M  +  \int_0^x|J(\ttau)|_M m(\ttau)
\\&\le
C \|J(\ttau) - J_0\|_1  +  \int_0^x|J(\ttau)|_M m(\ttau) ,
}
so by Gronwall's inequality,
$
|m(\ttau)|_0 \le C \|J(\ttau) - J_0\|_1 ,
$
and hence, by \eqref{J_lim.eq},
$$
\lim_{\ttau \to 0} | Q(\ttau) - z_0 |_0 = 0 .
$$
This proves that the derivative $\eta_\ttau(0)$ exists and equals $z_0$.
Since equations
\eqref{psi_pm_dtau.eq}-\eqref{psi_pm_dtau_bc.eq}
are clearly equivalent to the first two components of \eqref{linpb},
this completes the proof of Proposition~\ref{deriv_psi0.prop}.
\hfill$\Box$

\end{document}